\documentclass{amsart}
\newtheorem{theorem*}{Theorem}[section]
\newtheorem{theorem}{Theorem}[section]
\newtheorem{lem}[theorem]{Lemma}
\newtheorem{cor}[theorem]{Corollary}
\theoremstyle{definition}

\newtheorem{example}[theorem]{Example}

\theoremstyle{remark}
\newtheorem{remark}[theorem]{Remark}

\numberwithin{equation}{section}

\newcommand{\C}{\mathbb{C}}

\begin{document}

\title{OMITTED VALUES AND DYNAMICS OF MEROMORPHIC FUNCTIONS}

\author{Tarakanta Nayak and Jian-Hua Zheng}
 \address{Department of Mathematical Sciences,\\
   Tsinghua University, Beijing-100084,\\
   People's Republic of China.}
       \email{jzheng@math.tsinghua.edu.cn}
\email{tnayak@iitbbs.ac.in}
\thanks{The first author is supported by China Postdoctoral
Science Foundation (Grant no. 20090450420) and is now on leave from
National Institute of Technology, Rourkela (India).\\
  The second
author is partially supported by National Natural Foundation of
China (Grant No. 10871108).}



\date{November 8, 2012.}



\begin{abstract}
Let $M$ be the class of all transcendental meromorphic functions
$f~:~ \mathbb{C} \to \mathbb{C} \bigcup\{\infty\}~$ with at least
two poles or one pole that is not an omitted value, and $M_o =\{f
\in M~:~f~\mbox{has at least one omitted value}\}$. Some dynamical
issues of the functions in $M_o$ are addressed in this article. A
complete classification in terms of forward orbits of all the multiply connected Fatou
components is made. As a corollary, it follows that the Julia set is
not totally disconnected unless all the omitted values are
  contained in a single Fatou component. Non-existence of both Baker
wandering domains and invariant Herman rings are proved.
  Eventual connectivity of each wandering domain is proved to exist.
For functions with exactly one pole,
 we show that Herman rings of period two also do not exist.
  A necessary and sufficient condition
 for the existence of a dense subset of singleton buried components in the Julia set
  is established for functions with two omitted values. The conjecture that a
 meromorphic function has at most two completely invariant Fatou
 components is confirmed for all $f \in M_o$ except in the case when $f$ has a single omitted value,
 no critical value and is of infinite order. Some relevant examples are discussed.
\end{abstract}

\maketitle
%

\section{Introduction}
\label{intro}

Let $f: \mathbb{C} \rightarrow \widehat{\mathbb{C}}=\mathbb{C} \cup
\{ \infty \}$ be a transcendental meromorphic function. The set of
points $z \in \widehat{\mathbb{C}}$ in a neighborhood of which the
sequence of iterates $\{f^n(z)\}_{n=0}^{\infty}$ is defined and
forms a normal family is called the Fatou set of $f$ and is denoted
by $\mathcal{F}(f)$. The Julia set, denoted by $\mathcal{J}(f)$, is
the complement of $\mathcal{F}(f)$ in $\widehat{\mathbb{C}}$. It is
well-known that the Fatou set is open and the Julia set is a perfect
set. A component of $\mathcal{F}(f)$, to be called as a
\textit{Fatou component}, is mapped into a component of
$\mathcal{F}(f)$. For a Fatou component $U$, $U_k$ denotes the Fatou
component containing $f^k(U)$ where we take $U_0=U$ by convention. A
Fatou component $U$ is called $p$-periodic if $p$ is the least
natural number satisfying $U_p = U$.
 We say $U$ is invariant if $p=1$. An invariant component $U$ is called completely invariant
 if $f^{-1}(U) =U$. If $U$ is not periodic but $U_n$ is periodic for
some natural number $n$, then $U$ is called pre-periodic. A Fatou
component is called wandering if it is neither periodic nor
pre-periodic. A periodic Fatou component of a meromorphic function
is one of the five types, namely an attracting domain, parabolic
domain, Herman ring, Siegel disk or Baker domain. A Siegel disk or a
Herman ring is not completely invariant by definition. A detailed
description can be found in \cite{berg93}. Each limit function of
$\{f^n\}_{n>0}$ on $U$ is a constant if $U$ is an attracting or a
parabolic domain whereas it is a nonconstant function if $U$ is a
Herman ring or a Siegel disk. On a Baker domain, each limit function
of $\{f^n\}_{n>0}$ is either infinity or one of its pre-images. The
sequence $\{f^n\}_{n>0}$ can have infinitely many subsequential
limits (which are all constants in $\widehat{\mathbb{C}}$) on a
wandering domain. The connectivity of a periodic Fatou component is
known to be $1$, $2$ or $\infty$. An invariant Fatou component is
doubly connected if and only if it is a Herman ring. It is not known
whether a doubly connected periodic Fatou component of period
greater than $1$ is always a Herman ring \cite{bolsch}. A
pre-periodic Fatou component can have any finite connectivity
\cite{bk3}.
\par
 Let $M$ be the
class of transcendental meromorphic functions $f~:~\mathbb{C} \to
\widehat{\mathbb{C}}$ such that $f$ has either at least two poles or
exactly one pole that is not an omitted
 value.  These functions are usually referred to in the literature as
\textit{general meromorphic functions}. The backward orbit of
$\infty$ is an infinite dense subset of the Julia set in this case.
Let $O_f$ denote the set of all omitted values of $f$. Then $O_f$
has at most two elements and for each $w \in O_f$, there is no
ordinary point lying over $w$ and each singularity of $f^{-1}$ lying
over $w$ is necessarily transcendental (in fact direct). The
definition and classification of singularities of
 inverse function of a transcendental meromorphic function can be found in \cite{bergere95}.
 \\
 Let $$M_o =\{f \in M~:~O_f \neq
\emptyset\} ~\mbox{and}~M_o^{k} =\{f \in M~:~O_f~\mbox{has}~
k~\mbox{elements} \}~\mbox{for}~k \in \{1,~2\}. $$ For $f \in
M_o^2$, let $O_f =\{a,~b\}$. Then for a mobius map $T$ with $T(a)=1$
and $T(b)=-1$, $T(f) \in M_o^2$ and $O_{T(f)}=\{1,~-1\}$. Now,
$\frac{T(f)+1}{T(f)-1}$ is a transcendental entire function omitting
$0$ and can be written as $e^{2g}$ for some entire function $g$.
Thus, $T(f)=\frac{1}{\tanh(g)}$ and $f(z)=S(\tanh(g(z)))$ for the
mobius map $S(z)=T^{-1}( \frac{1}{z})$. Note that $S(1)=T^{-1}(1)=a
\neq \infty$ and $S(-1)=T^{-1}(-1)=b \neq \infty$ which gives that
$S^{-1}(\infty) \neq \pm 1$. A point $z$ is a pole of $f$ iff
$\tanh(g(z)) =S^{-1}(\infty)$. Since $f$ is a meromorphic function
with two distinct finite omitted values,  Picard's theorem implies
that $f$ has infinitely many poles whenever $f \in M_o^{2}$.
Similarly, a function in $ M^{1}_o$ can be written as
$\frac{1}{h(z)}+ a$ where $a$ is the omitted value of $f$ and
$h(z)=\frac{1}{f(z)-a}$ is an entire function. In this case, $f$ can
have finitely many poles.
\par
Singular values are well-known to be related in an important way to the
dynamics of a meromorphic function. For a transcendental meromorphic
function, omitted values are a special kind of singular value. The
significance of these values to some dynamical issues is
investigated in this article. It is shown that the local dynamics
of a transcendental meromorphic function at its omitted values
determine the topology of its Fatou set and hence of its Julia set
in considerable detail. This in turn leads to several useful
 conclusions. A multiply connected Fatou component of a
transcendental entire function is known to be a Baker wandering
domain \cite{b5}. A transcendental meromorphic function with exactly
one pole that is an omitted value has at most one multiply connected
Fatou component and this must be doubly connected \cite{b4}.
Multiply connected wandering domains for functions in $M_F=\{f \in
M~:~f ~\mbox{has at most finitely many poles}\}$ were discussed in
Zheng \cite{Zheng} and Rippon et al. \cite{rippon-stallard-5}.
However, multiply connected Fatou components of general meromorphic
functions are not restricted in general and this is a characteristic
departure from all the earlier cases.
\par
Our earlier discussion shows that $M_o \bigcap M_F \neq \emptyset$
and $M_o \setminus M_F \neq \emptyset$. This article examines
multiply connected Fatou components for functions in the class $M_o$
and establishes a complete classification in terms of the behavior
of their forward orbits. It is shown that such a Fatou component is
either wandering or lands  only on some special types of components,
namely a Herman ring, infinitely connected
 Baker domain or a Fatou component containing all the omitted values of the function.
As a corollary, it follows that the Julia set of $f$ is not totally disconnected unless all the
omitted values are contained in a single Fatou component. Some additional results on the dynamics
 of functions in $M_o$ are also presented, the proofs of which partly depend on the key ideas
 of the proofs of the earlier results.
 Invariant Herman rings and Baker wandering
domains are shown to be non-existent for all functions in $M_o$.
 Eventual connectivity of each wandering domain is determined.
 For functions with exactly one pole, we show that Herman rings of period two also do not exist.
 A necessary and sufficient condition for the existence of a dense
subset of singleton buried components in the Julia set is proved for
$f \in M_o^2$. The conjecture that a meromorphic function has at
most two completely invariant Fatou components is confirmed for all
$f \in M_o$ except the case when $f$ has a single omitted value, no
critical value and is of infinite order.  Statements of all results
with brief motivation and  implications are given in Section~2.
Section~3 contains the proofs  of Theorems~(1)-(5). The rest of the
results are proved in Section~4. Some examples are discussed at
relevant places.
\par
 For a closed curve $\gamma$ in $\mathbb{C}$, let $B(\gamma)$ denote
the union of all the bounded components of $\widehat{\mathbb{C}}
\setminus \gamma$. For a domain $D$ in $\widehat{\mathbb{C}}$, we
denote its boundary and connectivity by $\partial D$ and $c(D)$
respectively. By a component of the Julia set we mean a maximally
connected subset of the Julia set. We denote the component of the
Julia set containing a point $z \in \widehat{\mathbb{C}}$ (or a set
$A$) by $\mathcal{J}_z$ (or $\mathcal{J}_A$). For a set $A$, let
$|A|$ denote the number of elements in $A$, $A^c$ its complement in
$\widehat{\mathbb{C}}$ and $\overline{A}$ its closure in
$\widehat{\mathbb{C}}$. A Fatou component $U$ is said to land on a
Fatou component $V$ if $U_n =V$ for some $n$. Unless specifically
mentioned, by saying `for all $n$', we mean `for all $n \in
\mathbb{N} \bigcup \{0\}$' throughout this article.

\section{Results}
A classification of all multiply connected Fatou components of $f
\in M_o$ is made in the sense that each such component is wandering
or lands on a Fatou component $U$ where $U$ contains all the omitted
values of $f$,
 or $U$ is either a Herman ring or an infinitely connected
Baker domain of period greater than $1$.
 Precise situations leading
to these possibilities are the content of
Theorems~(\ref{sc1})-(\ref{singleton-buried}).
\par We say a Fatou
component $V$ is  SCH if one of the following holds.
\begin{enumerate}
 \item $V$ is simply connected.
\item $V$ is multiply connected with $c(V_n)>1$ for all $n \in \mathbb{N}$ and
 $V_{\bar{n}}$ is a Herman ring for some $\bar{n} \in
\mathbb{N} \bigcup \{0\}$.
\end{enumerate}
Clearly, $U$ is SCH implies $U_1$ is SCH whenever
$c(U)>1$.
\begin{theorem}
\label{sc1}
 Let $f \in M_o$ and $\mathcal{J}(f) \bigcap O_f \neq \emptyset$. If
$f \in M_o^2$ or $f \in M_o^1$ with $ |\mathcal{J}_{O_f}|>1$, then
each Fatou component of $f$ is SCH.
\end{theorem}
The next two results deal with the situation when $O_f$ intersects
the Fatou set. If the set $O_f$ intersects two Fatou components
$U_1$ and $U_2$, then exactly one of the following conditions holds: \\
(a) Both  $U_1$ and $U_2$ are unbounded,\\
 (b) Exactly one of $U_1$ and $U_2$
is unbounded, call it $U_1$ and $U_1$ is simply connected,\\
 (c)
Exactly one of $U_1$ and $U_2$ is unbounded, call it $U_1$ and $U_1$
is multiply connected with $U_2$ contained in the unbounded
component of $ U_1^{c}$, \\
(d) Exactly one of $U_1$ and $U_2$ is unbounded, call it $U_1$ and
$U_1$ is multiply connected with $U_2$ contained in a bounded
component of $ U_1^{c}$,\\
 (e) Both  $U_1$ and $U_2$ are bounded.
\begin{theorem}
\label{twofatoucomponent} Suppose $f \in M_o$.
 Let the set $O_f$ intersect two distinct Fatou components
$U_1$ and $U_2$ of $f$. Then,
\begin{enumerate}
\item If either (a) or (b) is satisfied, then all the Fatou
components of $f$ are simply connected.
\item The possibility (c) cannot be true.
\item If either (d) or (e) is satisfied, then each Fatou component of $f$ is SCH.

\end{enumerate}
\end{theorem}

\begin{theorem}
\label{onefatoucomponent}
 Suppose $f \in M_o$.
 Let $O_f$ be contained in a Fatou component $ U$ and
$V$ be a Fatou component with $V_n \neq U$ for any $n
  \in \mathbb{N} \bigcup \{0\}$. Recall that `for all $n$' means `for all $n \in \mathbb{N} \bigcup
\{0\}$'.
  \begin{enumerate}
  \item If $U$ is unbounded, then $c(V_n)=1$ for all $n$.
  \item If $U$ is bounded, then $V$ is SCH.
  \item If $U$ is wandering, then $c(U_n)=1$ for all $n$.
\item Let $U$ be pre-periodic but not periodic.
If $U$ is unbounded, then $c(U_n)=1$ for all $n$. If $U$ is bounded,
then $U$ is SCH.
\item If $U$ is periodic, then $c(U_n)=1$ or $\infty$ for all $n$.

\end{enumerate}
\end{theorem}

\begin{example}
\label{onefatoucomponent-remark} Both the possibilities of
Theorem~\ref{onefatoucomponent} (v) can be true. To see it, consider
$\lambda \frac{z^m}{\sinh^m z}$ where $m$ or $\frac{m}{2}$ is an odd
natural number and $\lambda \in \mathbb{R} \setminus \{0\}$. It is
shown in \cite{tkmgpp-unboundedsingular} that the only omitted value
$0$ is in $\mathcal{F}(f_\lambda)$ for all $\lambda$. Further, a
critical parameter $\lambda^* >0$ is found such that
$\mathcal{F}(f_\lambda)$ is  connected and
$c(\mathcal{F}(f_\lambda))=\infty$ for all $|\lambda| < \lambda^*$
and for $|\lambda|> \lambda^*$, each component of
$\mathcal{F}(f_\lambda)$ is simply connected. Denoting the Fatou
component containing $0$ by $U$ we have for $|\lambda| < \lambda^*$,
$U=\mathcal{F}(f_\lambda)$ and $ c(U_n) = c(\mathcal{F}(f_\lambda))
= \infty$ for all $n$. For $|\lambda| > \lambda^*$, $U$ is a
$2$-periodic component and $ c(U_n) =1$ for all $n$.
\end{example}
\begin{remark}
The Fatou component $V$ as assumed in the
Theorem~\ref{onefatoucomponent} may not always exist. This is
illustrated in the above example where each Fatou component
eventually lands on a periodic Fatou component ($1$-periodic for
$|\lambda|< \lambda^*$ and $2$-periodic otherwise) that contains all
the omitted values of the function.
\end{remark}
The residual Julia set of $f$, denoted by $\mathcal{J}_r(f)$, is
defined as the set of all those points in  $\mathcal{J}(f)$ that do
not belong to
 the boundary of any Fatou component. As observed by Baker and Dom\'{i}nguez~\cite{res-dom-fagella},
  this set is also residual in the sense of category theory.
A component of $\mathcal{J}(f)$ that is contained in $\mathcal{J}_r(f)$ is called a buried  component.
 For $|O_f|=1,~ O_f \bigcap \mathcal{J}(f)\neq \emptyset$ and $ |\mathcal{J}_{O_f}|=1$, the Fatou set has at
 least one multiply connected component. The next two theorems deal with all
 the multiply connected Fatou components in this situation.
 We say a point $z$ is a pre-pole if $f^{n}(z)= \infty$ for some $n \in \mathbb{N}$.
  Recall that for two Fatou components $U$ and $V$, $U$
 is said to land on $V$ if $U_n=V$ for some $n \in \mathbb{N}\bigcup
 \{0\}$.

\begin{theorem}
\label{singleton-notburied} Let $f \in M_o^1$, $O_f=\{a\} \subset
\mathcal{J}(f)$
 and $|\mathcal{J}_{a}|=1$. If $\mathcal{J}_{a}$ is not a buried component of
the Julia set, then $f$ has an infinitely connected Baker domain $B$ with period $p >1$
and $a$ is a pre-pole. Further, for each multiply connected Fatou
component $U$ of $f$ not landing on any Herman ring, there is a
non-negative integer $n$ depending on $U$ such that $U_n=B$. In this
case, singleton buried components are dense in $\mathcal{J}(f)$.
\end{theorem}
\begin{cor}
Let $f \in M_o^1$, $O_f=\{a\} \subset \mathcal{J}(f)$
 and $|\mathcal{J}_{a}|=1$. Then,
 \begin{enumerate}
 \item $\mathcal{J}_{a}$ is not
buried if and only if $a$ is a pre-pole.
\item $f$ has no completely invariant Fatou component.

 \end{enumerate}
 \label{cor-singleton}
\end{cor}

\begin{proof}
\begin{enumerate}
 \item
If $\mathcal{J}_{a}$ is not buried then $a$ is a pre-pole by
Theorem~\ref{singleton-notburied}. Conversely, let $a$ be a
pre-pole. Then $\mathcal{J}_{\infty}$ is a buried component of the
Julia set whenever $\mathcal{J}_{a}$ is. Suppose
$\mathcal{J}_{a}$ is a buried component. By taking a continuum $c $
in $\mathcal{J}(f)$ with sufficiently small diameter such that it
separates $\mathcal{J}_{a}$ from $\infty$ we can see that
$f^{-1}(c)$ has an unbounded component which must be in the Julia
set. However, it is not possible as $\mathcal{J}_{\infty}$ is a
buried component. Thus $\mathcal{J}_a$ is not buried.
\item   If $f$ has a completely invariant Fatou component $U$, then
 $\mathcal{J}_a
\subset \mathcal{J}(f) = \partial U$, which means that
$\mathcal{J}_a$ is not buried and $c(U)>1$. By
Theorem~\ref{singleton-notburied}, $U$ must land on a Herman ring or
on a Baker domain of period greater than $1$ which contradicts the
complete invariance of $U$. To see this, note that a Herman ring is
not completely invariant by definition. Hence, the claim follows.
\end{enumerate}
\end{proof}

\begin{theorem}
\label{singleton-buried} Let $f \in M_o^1$, $O_f=\{a\} \subset
\mathcal{J}(f)$ and $|\mathcal{J}_{a}|=1$. If $\mathcal{J}_{a}$ is a
buried component of the Julia set, then all the multiply connected
Fatou components not landing on any Herman ring are wandering and
$a$ is a limit point of
  $\{f^n\}_{n>0}$ on each of these wandering domains. Further, if $\mathcal{F}(f)$ has a  multiply connected
   wandering domain, then the forward orbit of $a$ is an infinite set and singleton buried components are dense
    in $\mathcal{J}(f)$.
\end{theorem}
For functions with two omitted values, if at least one omitted value is in the Julia set
 then each multiply connected Fatou
 component (if it exists) ultimately lands on a Herman ring. The same is also true when a function has only one
  omitted value and the component of the Julia set containing this value is nonempty and not a singleton
  (see Theorem 1).
  On the other hand,
  if there is only one omitted value of a function and the component of the Julia set containing this value is
  a singleton, then each multiply connected Fatou component is either wandering or eventually becomes
  a Herman ring or an infinitely connected Baker domain of period greater than 1.
A multiply connected Fatou component (if it exists)  ultimately lands
on a Herman ring or on a Fatou component containing all the omitted
values whenever all the omitted values are in the Fatou set. As
evident from the aforementioned theorems, there are situations in
which multiply connected Fatou components do not really occur.
\par
The conclusion of the next result is known for all transcendental
meromorphic functions with finitely many poles \cite{transdom98}.

\begin{cor}
 \label{nottotallydisconnected}
Let $f \in M_o$. If $O_f$ is not contained in a single Fatou
component of $f$, then  $\mathcal{J}(f)$ is not totally
disconnected.
\end{cor}

\begin{proof}

If all the Fatou components are simply connected, then
$\mathcal{J}(f)$ is connected and the claim follows trivially.
Suppose $\mathcal{F}(f)$ has at least one multiply connected
component $U$. Then $U$ is either a wandering domain or lands on $V$
where $V$ is a Herman ring or a Baker domain of period greater than
$1$ by Theorems~(\ref{sc1}), (\ref{twofatoucomponent}),
(\ref{singleton-notburied}) and (\ref{singleton-buried}). Therefore,
$U$ cannot be completely invariant and $\mathcal{F}(f)$ has at least
one component different from $U$. The boundary of $U$ has a
non-singleton component and thus, $\mathcal{J}(f)$ is not totally
disconnected.
\end{proof}
For a function  $f \in M_o$ not satisfying the assumption of the
above corollary, everything is possible regarding the connectivity
of the Julia set as described by the following examples.
\begin{example}
\label{nottotallydisconnected-rem} Let $M_o^*$ be the class of all
meromorphic functions $f$ in $M_o$ with $O_f$ contained in a single
Fatou component of $f$. Then,
\begin{enumerate}

\item There exists a function $f_1 \in M_o^*$ such that
$\mathcal{J}(f_1)$ is disconnected but not totally disconnected.
This can be seen by taking $f_1(z)=\lambda (e^z +1+\frac{1}{e^z
+1})$, $0< \lambda < \lambda^*$ where $\lambda^*$ is as defined in
 \cite{tk-thesis}. This function has a single omitted value $2
\lambda$ and it is proved in \cite{tk-thesis} that the Fatou set is
an infinitely connected attracting domain containing $2 \lambda$ and
the Julia set is not totally disconnected.

\item There exists a function $f_2 \in M_o^*$ such that
$\mathcal{J}(f_2)$ is totally disconnected. This can be seen by
taking $f_2(z)=\lambda \tan z$, $-1 < \lambda < 1$. In this case,
$\mathcal{F}(f_2)$ is connected and contains $O_{f_2}=\{i
\lambda,~-i \lambda\}$ but $\mathcal{J}(f_{\lambda})$ is a totally
disconnected set \cite{tanlkjk97}.

\item There exists a function $f_3 \in M_o^*$ such that
$\mathcal{J}(f_3)$ is connected. By taking $f_3(z)=\lambda
\frac{z^m}{\sinh^m}$, $m$ or $\frac{m}{2}$ is an odd natural number
and $\lambda$ is any non-zero real number, it is observed that
$O_{f_3}=\{0\}$. A critical parameter $\lambda^*>0$ is found in
 \cite{tkmgpp-unboundedsingular} such that for $|\lambda|>
\lambda^*$, $\mathcal{F}(f_3)$ is the basin of attraction or
parabolic basin corresponding to a $2$-periodic point and $0 \in
\mathcal{F}(f_3)$. Further, it is proved that all the Fatou
components are simply connected which means that $\mathcal{J}(f_3)$
is connected.
\end{enumerate}

\end{example}
 The following result on singleton components of the Julia set is
 proved by Dom\'{i}nguez in \cite{transdom98}.
\begin{theorem*}[A]
 Let $f$ be a transcendental meromorphic
function satisfying one of the following.
\begin{enumerate}
\item
 $\mathcal{F}(f)$ has a component with connectivity at least
$3$.
\item $\mathcal{F}(f)$ has three doubly connected components
$U_i,~i=1,~2,~3$ such that either (a) each component lies in the
unbounded component of the complement of the other two or (b) two of
the components $U_1,~U_2$ lie in the bounded component of $U_3^{c}$
but $U_1$ lies in the unbounded component of $U_2^{c}$ and $U_2$
lies in the unbounded component of $U_1^{c}$.
\end{enumerate} Then,
singleton components are dense in $\mathcal{J}(f)$.
\label{singleton-dense-dominguez}
\end{theorem*}
 Ng et al. \cite{zheng8} proved a generalization as follows.
\begin{theorem*}[B]
 Let $f$ be a meromorphic function that is not of the
form $\alpha+(z-\alpha)^{-k}e^{g(z)}$, where $k$ is a natural
number, $\alpha$ is a complex number and $g$ is an entire function.
Then $\mathcal{J}(f)$ has buried components if $f$ has no completely
invariant Fatou components and its Julia set is disconnected. Moreover, if
$\mathcal{F}(f)$ has an infinitely connected component, then the
singleton buried components are dense in $\mathcal{J}(f)$.
\label{singleton-buried-dense-zheng}
\end{theorem*}
 Using the above two results, we
give a necessary and sufficient condition for existence of singleton
buried components in $\mathcal{J}(f)$ for $f \in M_o^2$.

 \begin{theorem}
\label{singleton-buried-dense-tk} Let $f \in M_o^2$. Then, the
singleton buried components are dense in $\mathcal{J}(f)$ if and
only if $\mathcal{F}(f)$ has no completely invariant component and
$\mathcal{J}(f)$ is disconnected.
\end{theorem}
\begin{remark}
\begin{enumerate}
\item The above
result is similar to the so-called Makienko's
 conjecture, which states that the residual Julia set of a
rational function of degree at least two is empty if and only if the
Fatou set of
 $f$ has a completely invariant component or consists of only two components.
\item The proof  of
Theorem~\ref{singleton-buried-dense-tk} shows that the Julia set
contains singleton components whenever it is disconnected. This is
the conclusion of Theorem~A for all meromorphic functions when the
Julia set is disconnected in some specific ways not covering all
possibilities.

\item
Suppose $f \in M_o^2$ and $\mathcal{F}(f)$ has a Herman ring.
Existence of a completely invariant Fatou component $U$ would imply
that each Fatou component other than $U$ is simply connected. A
Herman ring is not completely invariant by definition and hence is
different from $U$. Thus, the Herman ring must be simply connected
which is a contradiction and we conclude that $\mathcal{F}(f)$ has
no completely invariant component. Since $\mathcal{J}(f)$ is
disconnected, by Theorem~\ref{singleton-buried-dense-tk}, singleton
buried components are dense in $\mathcal{J}(f)$ whenever $f \in
M_o^2$ and $\mathcal{F}(f)$ has a Herman ring.

\end{enumerate}
\end{remark}

\begin{cor}
Let $f \in M_o^2$. If $\mathcal{F}(f)$ has a completely invariant
component $V$ and $\mathcal{J}(f)$ is disconnected, then $V$ is the
only multiply connected Fatou component and $C(V)= \infty$. Further,
$O_f \subset V$.
 \label{bothomittedvaluesinU}
\end{cor}
\begin{proof}
A disconnected Julia set implies the existence of at least one
multiply connected Fatou component. The first part of the corollary
is a consequence of the fact that if $\mathcal{F}(f)$ has a
completely invariant component then each of its other components, if
such exist, is simply connected. Now, $c(V)=\infty$ follows from the
fact that $V$ is not a Herman ring. If $O_f$ intersects the Julia
set, then $V$ is SCH by Theorem~\ref{sc1}.
Theorem~\ref{twofatoucomponent} gives that $V$ is SCH whenever the
set $O_f$ intersects two Fatou components. Note that a multiply
connected completely invariant Fatou component, in particular $V$,
cannot be SCH. Thus,  $O_f $ is contained in a single Fatou
component, say $U$. If $U \neq V$ then $V_n=V \neq U$ for all $n$.
By Theorem~\ref{onefatoucomponent}(i) and (ii), either $c(V_n)=1$
for all $n$ or $V$ is SCH. As just observed, none of this can be
true. Therefore, $U=V$ and $O_f \subset U$ as desired.
\end{proof}
\begin{example}
\label{bothomittedvaluesinU-rem} The assumption of
Corollary~\ref{bothomittedvaluesinU} is not always true. In other
words, there are functions $g$ in $M_o^2$ for which $\mathcal{F}(g)$
has a completely invariant simply connected component and
consequently, has a connected Julia set. An example is $g(z)=\lambda
+\tan z,~ \lambda \in \mathbb{C}$, which has two omitted values
$\lambda + i$ and $\lambda - i$ and hence is in $M_o^2$. For each
$\lambda $ with $\Im(\lambda)>0$, it is seen that the upper half
plane is contained in a completely invariant attracting domain, say
$U$ \cite{tk-thesis}. By choosing $\lambda=i + \frac{\pi}{2}$ (any
other pole of $\lambda + \tan z$ can be taken in place of
$\frac{\pi}{2}$), we observe that $\frac{\pi}{2}$ is a pole as well
as an omitted value of $g$, i.e., $O_{g} \bigcap \mathcal{J}(g) \neq
\emptyset$. Applying Theorem~\ref{sc1}, we have $c(U)=1$ and
consequently, the Julia set is connected.
\end{example}

Conjecturally, the number of completely invariant Fatou components
of a meromorphic function is at most two. This has been proved to be
true for rational functions, transcendental entire functions and
transcendental meromorphic functions of finite type (those with
finitely many singular values). In what follows, we confirm this for
most of the functions in $M_o$. Note that the functions in the class
$M_o$
 are not necessarily of finite type.

\begin{theorem}
 \label{cifc}
Let $M_{cv}=\{f \in M~:~\mbox{ $f$ has at least one critical
value}\}$ and $CIFC_f$ denote the set of all completely invariant
Fatou components of $f$.
\begin{enumerate}

\item If $f \in M_o^2$, then $|CIFC_f| \leq 2$.

\item If $f \in M_o^2 \bigcap M_{cv}$, then $|CIFC_f| \leq 1$.
\item If $f \in M_o^1 \bigcap M_{cv}$, then $|CIFC_f| \leq 2$.
\end{enumerate}
\end{theorem}
\begin{remark}
\label{cifc-rem}
  If $f \in M_o^{1} \setminus
 M_{cv}$ is of finite order, then a result
 of Bergweiler et al. \cite{bergere95} guarantees that it has at most
 finitely many asymptotic values. Consequently, $f$ is of finite type and the number of
  completely invariant Fatou components is at most two. The other case
when $f$ is of infinite order remains open.
\end{remark}
\begin{example}
\begin{enumerate}
\item
 In Theorem~\ref{cifc}(i), $|CIFC_f|$ can be $0$ ,
  $1$ or $2$. As shown in \cite{tanlkjk97}, for all $\lambda > 1$, $\mathcal{F}(\lambda \tan
z)$ is the union of two completely invariant components, namely the
upper and the lower half planes. The other two possibilities hold
for the function $\lambda \tanh(e^z)$ for suitable values of
$\lambda$. This function has two omitted  values $\lambda $ and
$-\lambda$ and hence is in $M_o^2$. A critical parameter $\lambda^*
>0$ is found in \cite{tkmgpp-tanhexp} such that $\mathcal{F}(\lambda
\tanh(e^z))$ is the basin of attraction or parabolic basin
corresponding to a $2$-periodic point for $\lambda < \lambda^*$. In
this case, each periodic component has at least a pre-image
different from itself and therefore, $\mathcal{F}(\lambda
\tanh(e^z))$ has no completely invariant component. For $\lambda
>\lambda^*$, it is proved that $\mathcal{F}(\lambda \tanh(e^z))$ is
a completely invariant attracting domain.

\item
 There
exists a meromorphic function without any omitted value but with
critical values such that it has two completely invariant Fatou
components. For example, the upper and lower half planes are
completely invariant Fatou components for $z +\tan z$. In this case,
there are no omitted values and all the critical values are in the
Fatou set. Thus Theorem~\ref{cifc} (ii) is not true in general for
functions without omitted values.

\end{enumerate}
\end{example}

It has been proved that multiply connected Fatou components are
Herman rings or their pre-images in most cases. We prove mild
restrictions on the possibility of these domains.
\begin{theorem}
Let $f \in M_o$. If $H$ is a $p$-periodic Herman ring of $f$, then
 the bounded component of $H^c$ contains an essential
singularity of $f^p$. In particular, $\mathcal{F}(f)$ has no
invariant Herman ring. Further, if $f$ has only one pole then Herman
rings of period $2$ do not exist.
 \label{hermanring-period}
\end{theorem}
A wandering domain $U$ is called Baker wandering if for large enough
$n$, $U_n$ is bounded, multiply connected and surrounds $0$ such
that $U_n \to \infty$ as $n \to \infty$.  Given any path
$\gamma(t)~:~[0,~\infty) \to \mathbb{C}$ with $\lim_{t \to \infty}
\gamma(t)=\infty$, $\gamma$ intersects $U_n$ for all large $n$ where
$U$ is any Baker wandering domain of $f$. Consequently, $\lim_{t \to
\infty} f(\gamma(t)) $ cannot be finite. This rules out the
possibility of a finite asymptotic value and in particular, any
 omitted value for $f$. Thus, we have proved
\begin{theorem}
If $f \in M$ has a finite asymptotic value, then $\mathcal{F}(f)$ has
no Baker wandering domain. In particular, this is the case for all
$f \in M_o$. \label{nobakerwanderindomain}
\end{theorem}
Theorem~\ref{singleton-buried} gives a possibility for existence of
 multiply connected wandering domains for $f \in M_o$. Eventual connectivity of a wandering
domain $W$ of a meromorphic function is said to exist if $c(W_n)=p$
for all sufficiently large $n$ and some $p \in \mathbb{N}$.
Existence of eventual connectivity in general is a question yet to
have a complete answer. Zheng obtained the following result.
\begin{theorem*} [C] [( {\cite[p.\,219]{zheng-book}})]
Let $f~:~\mathbb{C} \to \widehat{\mathbb{C}}$ be transcendental
meromorphic and $W$ be a wandering domain in $\mathcal{F}(f)$. Then,
exactly one of the following is true.
\begin{enumerate}
\item For each $n$, $c(W_n)=\infty$.

\item For all large $n$, $c(W_n)=1$ or $2$.
\item For all large $n$, $c(W_n)=p \geq 3$ and $p$ is independent of
$n$.  In this case, $f~:~W_n \to
   W_{n+1}$ is univalent for all large $n$.
\end{enumerate}
\end{theorem*}
As an application,  it was proved in \cite{zheng-book} that for an
entire function, the eventual connectivity of its wandering domains
is $2$ or $\infty$ if it is a Baker wandering domain and $1$
otherwise. This is proved by Rippon et al.
\cite{rippon-stallard-5} for meromorphic functions with finitely
many poles. Here we prove the following.
\begin{theorem}
Let $W$ be a wandering domain of $f \in M_o$. Then eventual
connectivity of $W$ ($ec(W)$ ) exists. More precisely,
\begin{enumerate}

\item If $f \in M_o^1$, $O_f=\{a\} \subset \mathcal{J}(f)$ and $\mathcal{J}_a$ is singleton and buried, then
$ec(W) \in \mathbb{N} \bigcup \{\infty\}$.
\item In all other cases, $ec(W)=1$.
\end{enumerate}

 \label{eventualconnectivity-tk}
\end{theorem}
\begin{remark}
For a meromorphic function with finitely many poles, a multiply
connected wandering domain $W$ is Baker wandering if and only if
$W_n$ is multiply connected for infinitely many values of $n$
\cite{rippon-stallard-5}. From this and
Theorem~\ref{nobakerwanderindomain}, we conclude that if a function
$f \in M_o$ has finitely many poles and has a wandering domain $W$,
then $c(W_n)=1$ for all but finitely many values of $n$. In other
words, eventual connectivity of each wandering domain is one.
\end{remark}

\section{Proofs of Theorems~\ref{sc1},~\ref{twofatoucomponent}, ~\ref{onefatoucomponent},
~\ref{singleton-notburied} and ~\ref{singleton-buried}}

The following lemma concerning general meromorphic functions is useful
for our purposes.
\begin{lem} Let $f \in M$ and $V$ be a multiply connected  Fatou component of $f$. Also suppose that
$\gamma$ is a closed curve in $V$ with $B(\gamma) \bigcap
\mathcal{J}(f) \neq \emptyset$. Then there is an $n \in \mathbb{N}
\bigcup \{0\}$ and a closed curve $\gamma_n \subseteq f^n(\gamma)$
in $V_n$ such that $B(\gamma_n)$ contains a pole of $f$.
  Further, if $O_f \neq \emptyset$, then $O_f \subset B(\gamma_{n+1})$
  for some closed curve $\gamma_{n+1}$ contained in $f(\gamma_n)$.
\label{general1}
\end{lem}
\begin{proof}
Since $f \in M$ and $B(\gamma)\bigcap  \mathcal{J}(f) \neq
\emptyset$, there exists a $z \in B(\gamma)$ satisfying
$f^k(z)=\infty$ for some $k \in \mathbb{N}$. The set $\mathcal{N}=\{
m \in \mathbb{N}~:~ f^{m}(z)=\infty~\mbox{for some }~ z \in
B(\gamma) \}$ is a non-empty subset of $\mathbb{N}$ and  the Well-Ordering Property
 of natural numbers gives that $\mathcal{N}$ has a smallest element. Let it be
$\tilde{n}$ and set $n=\tilde{n}-1$. Then $n \in  \mathbb{N} \bigcup
\{0\}$ and $f^{n}~:~B(\gamma)   \to
 \mathbb{C}$   is analytic.   Hence,  $\gamma_n = \partial (f^{n}(B(\gamma)))$
 is a  closed curve contained in $V_n$ with
$\gamma_n \subseteq f^{n}(\gamma)$ and $B(\gamma_{n})$ contains a
pole of $f$.
\par
Suppose that the closure of $f(B(\gamma_n))$ contains an  element
$a$ of $O_f$. Let $\{w_k\}_{k>0}$ be a sequence in $f(B(\gamma_n))$
converging to $a$ and for each $k$, let $z_k$ be a point  in
$B(\gamma_n)$ satisfying $f(z_k)=w_k$. Then, considering a limit
point $z_0$ of $\{z_{k}\}_{k>0}$ we observe that $z_0 \in
\overline{B(\gamma_n)}$. The continuity of $f$ at $z_0$ gives that
$f(z_0)=a$. This is a contradiction since $a$ is an omitted value.
Therefore, $O_f \subset \widehat{\mathbb{C}} \setminus
\overline{f(B(\gamma_n))}$. The set $\overline{f(B(\gamma_n))}$ is
connected and contains a neighborhood of $\infty$. Consequently, $
\widehat{\mathbb{C}} \setminus \overline{f(B(\gamma_n))}$ is a
non-empty open set whose boundary is contained in $f(\gamma_{n})$
and there is a closed (and bounded but not necessarily simple) curve
$\gamma_{n+1} \subseteq f(\gamma_n)$ such that $O_f \subset
B(\gamma_{n+1}) $.

\end{proof}
\begin{remark}
\begin{enumerate}
\item
 Lemma~\ref{general1} also implies that, if there is a multiply
connected Fatou component of $f$ for $f \in M$, then there exists a
pole such that the component of $ \mathcal{J}(f)$ containing it is
bounded . In other words, if each component of the Julia set $
\mathcal{J}(f)$ containing a pole of $f$ is unbounded, then all the
Fatou components of $f$ are simply connected.
\item
 It
follows from the proof of Lemma~\ref{general1} that $c(V_j)>1$ for
all $j \in \{1,~2,...,~n\}$ where $n$ is as given in
Lemma~\ref{general1}.
\item
The second part of the proof of the above lemma gives that $O_f \bigcap
\overline{f(B)} =\emptyset$ for every bounded domain $B$.

\end{enumerate}
\label{allmultiplyconnected}
\end{remark}
Now, we present a lemma that will be used repeatedly.
\begin{lem} Let $f \in M_o$ and $V$
be a multiply connected Fatou component of $f$. Suppose  there are
two distinct numbers $c_1,~c_2 \in  \mathbb{C}$ such that for every
closed curve $\alpha$ in $\bigcup_{k \geq 0}V_k$ satisfying $O_f
\subset B(\alpha)$, we have $B(\alpha) \bigcap \mathcal{J}(f) \neq
\emptyset$ and $c_1,~c_2 \in B(\alpha)$. Then $c(V_n)>1$ for all $n$
and $V_{\bar{n}}$ is a Herman ring for some $\bar{n} \in \mathbb{N}
\bigcup \{0\}$. \label{general3}
\end{lem}
\begin{proof}
 Let $\gamma$ be a closed curve in $V$
such that $B(\gamma) \bigcap  \mathcal{J}(f) \neq \emptyset$. By
Lemma~\ref{general1}, there is an $\tilde{n} \in \mathbb{N} $ and a
closed curve $\gamma_{\tilde{n}}$ in $V_{\tilde{n}}$ with
$\gamma_{\tilde{n}} \subseteq f^{\tilde{n}}(\gamma)$ such that $O_f
\subset B(\gamma_{\tilde{n}})$. By assumption of this lemma, we have
$B(\gamma_{\tilde{n}}) \bigcap \mathcal{J}(f) \neq \emptyset$.
Setting $n_1=\tilde{n}$, we have a multiply connected Fatou
component $V_{n_1}$ of $f$ and a closed
 curve  $\gamma_{n_1}$  in $V_{n_1}$ with $\gamma_{n_1} \subseteq f^{n_1}(\gamma)$
  such that $B(\gamma_{n_1}) \bigcap  \mathcal{J}(f) \neq \emptyset$.
 Applying Lemma~\ref{general1} again to $\gamma_{n_1}$ and $V_{n_1}$, we can find $m
\in \mathbb{N} \bigcup \{0\}$ such that $O_f \subset
B(\gamma_{n_1+m+1})$
 where $ \gamma_{n_1+m+1} $ is a closed curve with
  $\gamma_{n_1+m+1} \subseteq f^{m+1}(\gamma_{n_1}) \subset V_{n_1+m+1}$. Set $n_2=n_1+m+1$
   and observe that $n_2>n_1$.
This argument can be repeated since $V_{n_2}$ is a multiply
connected
 Fatou component (which follows from assumption) containing a closed curve $\gamma_{n_2}$ such that
 $B(\gamma_{n_2}) \bigcap \mathcal{J}(f) \neq \emptyset$. An inductive argument gives
  rise to an increasing sequence $\{n_k\}_{k>0}$ such that $O_f
 \subset B(\gamma_{n_k})$ for each $k$. It is clear from Remark~ \ref{allmultiplyconnected} (ii) that
 $c(V_n)>1$ for $n \in \mathbb{N} \setminus \{n_1,~n_2,~n_3...\}$.
Since the above considerations give $c(V_n)>1$ for $n \in \{n_1,~n_2,~n_3,...\}$, we conclude that
 $c(V_n)>1$ for all $n$.
\par
Since the sequence $\{f^{n}\}_{n>0}$ is normal on $V$,
$\{f^{n_k}\}_{k>0}$ has a subsequence converging uniformly to a
function $g(z)$ on compact subsets of $V$. Without loss of
generality, we denote this subsequence by $\{f^{n_k}\}_{k>0}$.  Now,
if $g(z)$ is a constant function $c \in \widehat{\mathbb{C}}$ then
two cases arise as follows.
\\
\underline{\noindent{Case I:~$c\in  \mathbb{C}$}}\\
Since $\{c_1,~c_2\} \subset B(\gamma_{n_k})$ for each $k$, we can
choose a finite point $c'$ in $\{c_1,~c_2\} \setminus \{c\}$ that
is contained in $B(\gamma_{n_k}) $ for all $k $. Now, each ball
around $c$ with radius less than $|c-c'|$ will contain
$\gamma_{n_k}$ for all sufficiently large $k$. This gives that $c'
\notin B(\gamma_{n_k})$ for sufficiently large $k$ leading to a
contradiction.
\\
\underline{\noindent{Case II:~$c = \infty$}}\\
Now we have that $f^{n_k}|_{V} \to \infty$ uniformly on compact
subsets of $V$, and $\gamma_{n_k}\to \infty$ as $k \to \infty$.
Since $O_f\subset B(\gamma_{n_k})$, we can assume that, there is a
pole in each $B(\gamma_{n_k})$. Then, by Lemma~\ref{general1},
$O_f\subset B(\gamma_{n_k+1})$ where $\gamma_{n_k+1} \subseteq
f(\gamma_{n_k} )$ is a Jordan curve as given in the lemma. If
$\gamma_{n_k+1}$ has a finite limit then, after passing down to a
subsequence if necessary, a contradiction can be obtained as in Case
I. Therefore, $\gamma_{n_k+1} \to \infty$ and $f^{n_k+1}(\gamma)\to
\infty$ as $k \to \infty$. There exists a tract $U$ over a small
neighborhood of $a\in O_f$. Obviously, $\gamma_{n_k}\cap
U\not=\emptyset$ for each $k$ and $f(\gamma_{n_k}\cap U)$ is
contained in a small neighborhood of $a$ and at the same time, we
have $f(\gamma_{n_k}\cap U)\subset f^{n_k+1}(\gamma)$. This is a
contradiction.
 \par
 Thus, $f^{n_k}|_V$ converges uniformly to a non-constant function. In particular, $V$ is not a
 wandering domain and $V_{\bar{n}}$ is periodic for some $\bar{n} \in \mathbb{N} \bigcup
 \{0\}$. Since $c(V_n)>1$ for all $n$,
  $V_{\bar{n}}$ cannot be a Siegel disk and it must be a Herman ring as desired.
\end{proof}

\begin{remark}
\begin{enumerate}
\item
Let the second sentence of the assumption of the
Lemma~\ref{general3} be modified as ``Suppose  there is a complex
number $c$ such that for every closed curve $\alpha$ in $\bigcup_{k
\geq 0}V_k$ satisfying $O_f \subset B(\alpha)$, we have $B(\alpha)
\bigcap \mathcal{J}(f) \neq \emptyset$ and $c \in B(\alpha)$''.
Then, we still get $c(V_n)>1$ for all $n \in \mathbb{N}$. However,
it is not true in general that $V$ ultimately lands on a Herman ring
in this case.

 \item
 Following the arguments of
Lemma~\ref{general1}, it is seen that if $\mathcal{F}(f)$ has a
multiply connected component $V$ then a closed curve $\alpha$ exists
in $\bigcup_{k \geq 0}V_k$ such that $O_f \subset B(\alpha)$.
 \end{enumerate}
 \label{remark-lemma-general3}

\end{remark}

\begin{proof}[Proof of Theorem~\ref{sc1}]
Let $V$ be any multiply connected Fatou component. Also, let
$\alpha$ be a closed curve in $\bigcup_{k \geq 0}V_k$ such that $O_f
\subset B(\alpha)$. Such a closed curve exist in view of
Remark~\ref{remark-lemma-general3}(ii). Then $B(\alpha) \bigcap
\mathcal{J}(f) \neq \emptyset$ by assumption. If $f \in M_o^2$ then
choose $c_1,~c_2$ to be the omitted values of $f$. If $|O_f|=1$ and
$ |\mathcal{J}_{O_f}|>1$, then choose any two distinct points of $
\mathcal{J}_{O_f}$ as $c_1$ and $c_2$. Now, $c_1,~c_2 \in B(\alpha)$
and are independent of $\alpha$. Therefore, $V$ is SCH by
Lemma~\ref{general3}.

\end{proof}

\begin{proof}[Proof of Theorem~\ref{twofatoucomponent}]
\begin{enumerate}
\item Suppose $V$ is a multiply connected Fatou component of $f$ and
 $\alpha$ is a closed curve in $V$ such that $B(\alpha) \bigcap \mathcal{J}(f) \neq
 \emptyset$. By Lemma~\ref{general1}, there is an $n \in \mathbb{N}$ and a
  closed curve $\alpha_{n} \subseteq f^{n}(\alpha) \subset V_n$ such that $O_f \subset
  B(\alpha_{n})$. Since $U_i \bigcap O_f \neq \emptyset$ for $i=1,~2$, we have
   $U_i \bigcap   B(\alpha_{n}) \neq \emptyset$ for each $i$.
   Further, if $U_i$ is unbounded for some $i$, then  $U_i \bigcap \alpha_{n} \neq
   \emptyset$ and consequently, $U_i =V_n$.
 Let (a) hold. Then each $U_i$ is unbounded and we have $V_n = U_i$ for each $i$ which means that $U_1 =U_2$.
 This contradicts our assumption that $U_1 \neq U_2$.  Now, let (b) be true.
  Then $U_1$ is unbounded and simply connected. We have $U_1 = V_n$.
  Since $O_f \subset  B(\alpha_{n})$ and $U_1 \neq U_2$, it follows that $U_2 \subset B(\alpha_{n})$.
  Therefore, $U_1^c$ has a bounded component containing $U_2$. In particular, $c(U_1)>1$.
  But $c(U_1)=1$ is assumed in (b). Therefore, all
 the Fatou components of $f$ are simply connected whenever either (a) or (b) is satisfied.
 \item
 Let (c) be true. Then $U_1$ is a multiply connected unbounded component with $U_2$ contained in the
   unbounded component of $U_1^c$. Considering $U_1$ in place of $V$ and arguing similarly as in the first
   portion of the preceding case, we can get $U_{1+n}=U_1$ for
   some $n \in \mathbb{N}\bigcup \{0\}$ and $U_2$ contained in a bounded component of
   $U_1^c$. However, this contradicts the assumption made in (c).
   Thus, the possibility (c) cannot be true.
\item
 Suppose $V$ is a multiply connected Fatou component of $f$
and
 $\alpha$ is a closed curve in $\bigcup_{k \geq 0} V_k$ such
that $O_f \subset B(\alpha)$.  Such a closed curve $\alpha$ exists
in view of Remark~\ref{remark-lemma-general3}(ii). Then $U_i \bigcap
B(\alpha) \neq \emptyset$ for $i=1,~2$. Let (d) be satisfied.
Unboundedness of $U_1$ gives $U_1 \bigcap B(\alpha) \neq \emptyset$
and $\partial U_2 \subset B(\alpha)$. Similarly, if (e) is satisfied
then at least one element of $\{\partial U_1,~\partial U_2\}$, say
$\partial U_1$, is contained in $B(\alpha)$. Choose two distinct
points $c_1,~c_2$ in $B(\alpha) \bigcap \partial U_2$ if (d) is true
or in $B(\alpha) \bigcap
\partial U_1 $ if (e) is true. Then $c_1,~c_2 \in B(\alpha)$ and
$B(\alpha) \bigcap \mathcal{J}(f) \neq \emptyset$ for each closed
curve $\alpha$ in $\bigcup_{k \geq 0} V_k$ with $O_f \subset
B(\alpha)$. By Lemma~\ref{general3}, $V$ is SCH.
\end{enumerate}
\end{proof}
\begin{lem}
Suppose $f \in M_o$.  If $O_f \bigcap U \neq \emptyset$ for some
Fatou component $U$, then $c(U_1)=1$  implies $c(U)=1$.
\label{iffsc}
\end{lem}
\begin{proof}
Assume $U_1 \neq U$ because the proof is trivial for $U=U_1$.
Suppose $U$ is multiply connected and $\alpha$ is a closed curve in
$U$ with $B(\alpha) \bigcap \mathcal{J}(f) \neq \emptyset$. If
$B(\alpha)$ contains a pole of $f$, then there exists a closed curve
$\alpha_1 \subseteq f(\alpha) \subset U_1$ such that $O_f \subset
B(\alpha_1)$. It gives that $\partial U \subset B(\alpha_1)$ and
consequently, $B(\alpha_1) \bigcap \mathcal{J}(f) \neq \emptyset$.
If $B(\alpha)$ does not contain a pole of $f$, then $f$ is an
analytic function on $B(\alpha)$ and we have  $B(\alpha_1) \bigcap
\mathcal{J}(f) \neq \emptyset$. Thus, $c(U_1)>1$. In other words,
$c(U_1)=1$ implies $c(U)=1$.

\end{proof}
\begin{proof}[Proof of Theorem~\ref{onefatoucomponent}]
\begin{enumerate}
\item If $V_k$ is multiply connected for some $k \in \mathbb{N} \bigcup
\{0\}$, then by Lemma~\ref{general1}, there is a closed curve
$\alpha \subset V_{m}$ for some $m \geq k$ such that $O_f \subset
B(\alpha)$. Since $V_n \neq U$ for any $n \in \mathbb{N} \bigcup
\{0\}$ and $O_f \subset U$, we have $\partial U \subset B(\alpha)$.
However, this is not possible
 if $U$ is unbounded. Therefore, $U$ is unbounded implies $c(V_n)=1$
 for all $n$.
\item
If $\alpha$ is a closed curve in $\bigcup_{k \geq 0} V_k$ such that
$O_f \subset B(\alpha)$ then $\partial U \subset B(\alpha)$ since
$V_n \neq U$ for any $n \in \mathbb{N} \bigcup \{0\}$. This means that
$B(\alpha)$ intersects the Julia set and contains two points
$c_1,~c_2$ of $\partial U$ for any closed curve $\alpha$ in
$\bigcup_{k \geq 0} V_k$ with $O_f \subset B(\alpha)$. By
Lemma~\ref{general3}, $V$ is SCH.

\item Setting $V=U_1$, we observe that $V_n=U_{1+n} \neq U$ for any
 $n \in \mathbb{N} \bigcup
\{0\}$. If $U$ is unbounded, then $c(U_n)=1$ for all $n \in
\mathbb{N}$ by Theorem~\ref{onefatoucomponent} (i). If $U$ is
bounded, then $V=U_1$ is SCH by Theorem~\ref{onefatoucomponent}(ii).
In fact, the proof of Theorem~\ref{onefatoucomponent}(ii) gives that
$V_n=U_{1+n}$ is SCH whenever $U$ is bounded. Consequently, if
$c(V_n)>1$ for any $n \in \mathbb{N} \bigcup \{0\}$ then $V_{n^*}$
is a Herman ring for some natural number $n^*$. This is not possible
since $V$ is already assumed to be a wandering domain. Therefore,
$c(V_n)=1$ for all $n$ and as result, we get $c(U_n)=1$ for all
natural numbers $n$. Now, simple connectedness of $U_0=U$ follows
from Lemma~\ref{iffsc} and the proof is complete.

\item Observe that $U_k \neq U$ for $k \in \mathbb{N}$ since $U$ is pre-periodic but not periodic.
If $U$ is unbounded, then $c(U_k)=1$ for all $k \in \mathbb{N}$ by
Theorem~\ref{onefatoucomponent} (i). Lemma~\ref{iffsc} gives
$c(U)=c(U_0)=1$. If $U$ is bounded, then $c(U)>1$ implies $c(U_1)>1$
by Lemma~\ref{iffsc}. By Theorem~\ref{onefatoucomponent} (ii), $U_1$
is SCH which means $U$ is SCH.
\item The component $U$ cannot be a Herman ring as it contains at
least one omitted value. Therefore, $c(U_n)=1$ or $\infty$ for all
$n$.
\end{enumerate}
\end{proof}

Now, we present a lemma before giving the proof of
Theorem~\ref{singleton-notburied}. The proof of the lemma follows
trivially. Recall that $\mathcal{J}_z$ denotes the component of
$\mathcal{J}(f)$ containing $z$.
\begin{lem}
Let $f \in M_o^1$ and $O_f =\{a\} \subset \mathcal{J}(f)$. Suppose
$U$ is a Fatou component of $f$ and $\partial U$ contains a point $s
\in \widehat{\mathbb{C}}$ such that $\mathcal{J}_s$ is singleton.
Then $\mathcal{J}_{f(s)}$ is a singleton component of $\partial U_1$
where we take $f(\infty)=a$. In particular, $c(U_n) = \infty$ for
all $n$.
 \label{oneomittedvalue-hermanring}
\end{lem}

\begin{proof}[Proof of Theorem~\ref{singleton-notburied}]
 That the component $\mathcal{J}_a$ of the Julia set is a singleton and not buried means
 $\mathcal{J}_a \subset \partial U$ for some Fatou
component $U$ and $c(U)=\infty$. Taking a closed curve $\gamma$ in
$U$ with $B(\gamma) \bigcap \mathcal{J}(f) \neq \emptyset$ and
arguing similarly as in the first part of the proof of
Lemma~\ref{general3}, a sequence of closed curves
$\{\gamma_{n_k}\}_{k>0}$ can be found such that $\mathcal{J}_a
\subset B(\gamma_{n_k})$ for a closed curve $\gamma_{n_k} \subseteq
f^{n_k}(\gamma) \subset U_{n_k}$. Note that $c(U_{n_k})>1$ for each
$k$. Further, if $f^{n_k}|_{U}$ has a constant limit function, then
it can only be $\infty$ or $a$. In view of the arguments of Case-II
of the proof of Lemma~\ref{general3}, we can assume, without loss of
generality, that $f^{n_k}|_{U} \to a$ as $k \to \infty$ and
$c(U_{n_k})>1$ for each $k$.

\par
Suppose $U$ is a wandering domain. Then, $U_{n_k} \neq U_{n_{k'}}$
for $k \neq k'$ and it follows that $\mathcal{J}_a$ is a buried
component of the Julia set: a contradiction. Therefore, $U_p$ is a
periodic Fatou component for some $p$.

\par
Suppose $\{U_p,~U_{p+1},~U_{p+2},...,U_{p+(l-1)}\}$ is the
$l$-periodic cycle of Fatou components. Then, there is a subsequence
$\{n_{k(i)}\}_{i>0}$ of $\{n_k\}$ and some $t \in
\{p,~p+1,~p+2,...p+(l-1)\}$ such that $f^{n_{k(i)}}(U_t) \subseteq
U_t$. Further, if $U_p$ is an attracting domain or a parabolic
domain, then $f^{n_{k(i)}}|_{U_t} \to a$ as $i \to \infty$. Since $a
\in \mathcal{J}(f)$, $U_t$, and hence  $U_p$, is not an attracting
domain. Also, $U_p$ cannot be a parabolic domain because $a \subset
B(\gamma_{n_{k(i)}})$ for each $i$.
 We have already observed that $c(U_{n_k})>1$.  Consequently,
$U_p$ cannot be a Siegel disk. By
Lemma~\ref{oneomittedvalue-hermanring}, $c(U_{p})=\infty$ and hence
$U_p$ is not a Herman ring.  The only remaining case, which must be
true, is that $U_{p}$ is a Baker domain. That the period of $U_p$ is
at least two and $a$ is a pre-pole follow from the fact that $a$ is
a finite complex number and is a limit function of $f^n|_{U_p}$.
\par
For any multiply connected Fatou component $V$ not landing on a
Herman ring,  the above argument clearly shows that $V_n =B$ for
some $n \in \mathbb{N} \bigcup \{0\}$.
\par
Setting $B=U_p$, we observe that $c(B) =\infty$. Further,
$\mathcal{F}(f)$ has no completely invariant component since $B$ is
not itself completely invariant and any other completely invariant
component would imply $c(B)=1$. Now by Theorem~B, singleton buried
components are dense in the Julia set.

\end{proof}

 Now, we proceed to prove Theorem~\ref{singleton-buried}.

\begin{proof}[Proof of Theorem~\ref{singleton-buried}]
Let $W $ be any multiply connected Fatou component not landing on a
Herman ring. Following the argument of the proof of
Lemma~\ref{general3}, we observe that $\mathcal{J}_a \subset
B(W_{m_k})$ and $f^{m_k}|_{W } \to a$ for some subsequence
$\{m_k\}_{k>0}$. Now, $\mathcal{J}_a$ is buried gives that $W$ does
not land on a periodic Fatou component and hence, is wandering.
\par
Suppose that $\mathcal{F}(f)$ has a multiply connected wandering
domain. If the forward orbit of $a$ is finite then we can find a
subsequence $\{m_{k(i)}\}_{i>0}$ of $\{m_k\}_{k>0}$ such that
$f^{m_{k(i)}}|_{W} \to a$ as $i \to \infty$. Applying Proposition~1
of \cite{zheng1} to this situation we conclude that $a$ is a
pre-pole. But this is not possible by Corollary~\ref{cor-singleton}
(i). Thus, the forward orbit of $a$ is an infinite set. Evidently,
each point of the grand orbit of $a$ (this is the set of all points
$z$ satisfying $f^m(z)=f^n(a)$ for some $m,~n \in \mathbb{N}$) is a
singleton buried component of the Julia set which is clearly dense
in $\mathcal{J}(f)$.
\end{proof}
\section{Proofs of Theorems~\ref{singleton-buried-dense-tk},~\ref{cifc},
~\ref{hermanring-period} and ~\ref{eventualconnectivity-tk}}

 The following result is due to Bolsch \cite{bolsch} and is stated in a
 simpler form to suit our purpose.
\begin{theorem*}[D]
\label{bolsch}
 Let $f~:~\mathbb{C} \to \widehat{\mathbb{C}}~$ be a
transcendental meromorphic function. If $H \subset  \C$ is
a domain and $G$ is any component of $f^{-1}(H)$, then exactly one
of the following holds.
\begin{enumerate}

\item There exists $ n \in \mathbb{N}$ such that $f$ assumes in $G$
every value of $H$ exactly $n$ times. In this case, $c(G)-2=n
(c(H)-2)+v$ and $v \leq 2n-2$, $v$ denoting the number of critical
points of $f$ in $G$ counting multiplicities.
\item $f$ assumes in $G$
every value of $H$ infinitely often with at most two exceptions. In
this case, $c(H)>2$ implies $c(G)=\infty$.
\end{enumerate}
\end{theorem*}
In the above theorem, we say $G$ is an island of multiplicity $n$
over $H$ if (i) holds ( $f~:~G \to H$ is a proper map in this case).
If (ii) holds, then $G$ is said to be a tongue over $H$.

\begin{proof}[Proof of Theorem~\ref{singleton-buried-dense-tk}]

 If  singleton buried components are dense in
$\mathcal{J}(f)$, then it clearly follows that $\mathcal{J}(f)$ is
disconnected and $\mathcal{F}(f)$ has no completely invariant
components.
\par
Conversely, suppose $\mathcal{F}(f)$ has no completely invariant
components and $\mathcal{J}(f)$ is disconnected. If all the Fatou
components of $f$ are simply connected then $\mathcal{J}(f)$ is
connected, which is against our assumption. Therefore,
$\mathcal{F}(f)$ has at least one multiply connected component, say
$U$. In view of Lemma~\ref{general1}, we can find an $n$ such that
$U_n$ contains a closed and bounded curve $\gamma$ with $O_f \subset
B(\gamma)$. Let $B=B(\gamma)$ and $B_{-1}$ be a component of
$f^{-1}(B)$. Since $f$ does not take two values of $B$ in $B_{-1}$,
the map $f~:~B_{-1} \to B$ cannot be proper. By Theorem~D, $B_{-1}$
is a tongue over $B$. Then $B_{-1}$ is a component of $f^{-1}(B
\setminus O_f)$ and is a tongue over $B \setminus O_f$. Observing
that $c(B \setminus O_f)=3$, we have $c(B_{-1})= \infty$ by
Theorem~D(ii). Clearly, $B_{-1}$ is unbounded as each singularity
lying over an omitted value is transcendental.

\par
Let all the bounded components of the boundary of $B_{-1}$ be
enumerated by $\gamma_i$ for $ i \in \mathbb{N}$. For each $i$,
$f(\gamma_i) \subseteq \gamma$ gives that $f$ has no pole on
$\gamma_i$. Next, we assert that each $\gamma_i$ is a continuum
(compact and connected set) separating the plane and $B(\gamma_i)
\bigcap \mathcal{J}(f) \neq \emptyset$. If $\gamma_i$ does not
separate the plane for some $i$, then a Jordan curve $\alpha$ can be
found in $B_{-1}$ such that $\gamma_i \subset B(\alpha)$ and
$f~:~B(\alpha) \to \mathbb{C}$ is analytic. By the Maximum Modulus
Principle, $f(B(\alpha))$ contains $f(\gamma_i)$. Since $f(\gamma_i)
\subseteq \gamma$, $\gamma$ is connected and $\partial
(f(B(\alpha))) \subseteq f(\alpha)$ does not intersect $\gamma$, we
have $\gamma \subset f(B(\alpha))$. However, this is not possible
because $f(\alpha)$ is a closed curve in $B$,  $B$ is simply
connected and $\gamma = \partial B$. Therefore, each $\gamma_i$ is a
continuum
 separating the plane. If $f(\gamma_i)$ is properly
contained in $\gamma$ for some $i$ then $f(B(\gamma_i)) \supseteq
(\widehat{\mathbb{C}} \setminus O_f) \setminus f(\gamma_i)$ where
$B(\gamma_i)$ is the union of all bounded components of
$\widehat{\mathbb{C}} \setminus \gamma_i$. If $f(\gamma_i) =\gamma$
then $f(B(\gamma_i)) =\widehat{\mathbb{C}} \setminus  B$ as
$B(\gamma_i) \bigcap B_{-1} = \emptyset$. In any case,
$f(B(\gamma_i))$ contains $\infty$ and hence $f$ has a pole in
$B(\gamma_i)$, which proves our assertion that $B(\gamma_i) \bigcap
\mathcal{J}(f) \neq \emptyset$. Note that $B(\gamma_i) \subset
B(\gamma_j)$ is not possible for any $i \neq j$ as each $\gamma_i$
is a bounded component of $\partial B_{-1}$ and $B_{-1}$ is an
infinitely connected unbounded domain. Observe that each $\gamma_i$
is a continuum in the Fatou set and $\gamma_i^{c}$ has at least two
components intersecting the Julia set. This implies that either
$\mathcal{F}(f)$ has an infinitely connected component or has
infinitely many components, each of which is at least doubly
connected. This satisfies the assumption of Theorem~A(ii) and we
have that singleton components are dense in the Julia set. If any of
these components is not buried then an infinitely connected Fatou
component is found and it follows by ~Theorem~B that singleton
buried components are dense in the Julia set. This completes the
proof.
\end{proof}

 The lemma to follow relates completely invariant Fatou components
 with the omitted values of a meromorphic function and will be used for proving
 Theorem~\ref{cifc}. Recall that $M_o^{k}=\{f \in
M~:~|O_f|=k\}$ for $k \in \{1,~2\}$ and let $CIFC_f =\{V~:~V
~\mbox{is a completely invariant Fatou component of}~ f\}$.
\begin{lem}
\label{lemma-completelyinvariant}
 Let $M_{cv}=\{f \in M~:~\mbox{ $f$ has at
least one critical value}\}$.
\begin{enumerate}
 \item
 Let $f \in M_o^{2}$ and $U \in CIFC_f$.
 Then $c(U)=1$ if
 and only if $|U \bigcap O_f|=1$.
 \item
 Let $f \in M_o \bigcap M_{cv}$ and let $U \in CIFC_f$ be such that $c(U)=1$.
Then $CV_f \subset U$ or $O_f \subset U$ where $CV_f$ denotes the
set of all critical values of $f$. In particular, if $f \in M_o^2
\bigcap M_{cv}$, $U \in CIFC_f$ and $c(U)=1$, then $CV_f \subset U$.
\end{enumerate}
\end{lem}
\begin{proof}
 \begin{enumerate}
 \item
 Let $f \in M_o^{2}$ and $U \in CIFC_f$.
  Suppose that $c(U)=1$. First, we shall prove  $|O_f \bigcap U|>0$.
If this is not true, then $O_f =\{a,~b\} \subset U^c$.
  Take a point $ u \in U$ (which is neither a critical value nor an omitted value of $f$)
   in a neighborhood of which
 $f^{-1}$ has a well-defined analytic
 branch. Let this branch be $\phi$.  By the Gross Star Theorem~{\cite[Proposition\,1]{berg-direct}}, $\phi$
 can be continued analytically along
  a Jordan curve $\gamma_1$ passing through
 $u$ and winding around $a$ exactly once but not winding around $b$.
 The curve $\gamma$ defined by $
 \phi(\gamma_1)$ has its end points $u_1$ and $u_2$ in $U$, which can be joined by
 a simple curve $\beta$ entirely contained in $U$. This is possible since $U$ is
 completely invariant and path connected. Setting $A=B(\gamma \bigcup \beta)$, the bounded component of
   $(\gamma \bigcup \beta)^c$,
we observe that $\partial f(A) \subseteq \gamma_1
 \bigcup f(\beta)$. Further, $f(\beta) \subset U$ and $c(U)=1$ give
 that $f(\beta)$ winds  around neither   $a$ nor $b$.  This means $\gamma_1
 \bigcup f(\beta)$ winds around $a$ but not around $b$.
    Now, if $f$ is analytic in $A$ then $A \setminus U$ is mapped
    into $B(\gamma_1)\setminus U$ and, consequently, $a$  is
 in the closure of $f(A)$. This is not possible by
 Remark~\ref{allmultiplyconnected} (iii).
  Supposing $f$ has a pole in $A$, we observe that $b$ is
  in the closure of $f(A)$ which is also not possible because of
Remark~\ref{allmultiplyconnected} (iii). Thus, we conclude that
$|O_f \bigcap U|>0$. Note that $U$ is a tongue over $U \setminus
O_f$. If $|O_f \bigcap U|=2$ then $c( U \setminus O_f)=3$ and we
have $c(U)=\infty$ by Theorem~D(ii). This contradicts our initial
assumption that $c(U)=1$. Thus, $|O_f \bigcap U|=1$.
\par
Conversely, let $|O_f \bigcap U|=1$. If $c(U)>1$ then $O_f \subset
U$ by Corollary~\ref{bothomittedvaluesinU}, which means that $|O_f
\bigcap U|=2$ and our assumption is contradicted. Therefore, we
conclude that $c(U)=1$.

\item
Let $f \in M_o \bigcap M_{cv}$ and  let $U \in CIFC_f$ be
 such that $c(U)=1$. Suppose that $CV_f \subset U$ is not true.
 Then a critical value $c$ can be found in $ U^c$. As in (i), take $u \in U$
 such that it is neither a critical value nor an omitted value, an analytic branch
 $\phi$ of $f^{-1}$ defined in a neighborhood of $u$, a Jordan curve
 $\gamma_1$ passing through $u$ and winding around $c$ once but not around any omitted value such that
  $\phi$ can be continued analytically
 along $\gamma_1$. If $f$ is analytic in $A=B(\gamma \bigcup \beta)$,
 then $\gamma_1 \bigcup f(\beta)$ winds around the critical value
 $c$ at least twice.
 Since $c(U)=1$ and $f(\beta) \subset U$, $\gamma_1$ winds around $c$ at least
 twice. This is not true. Suppose $f$ has a pole in $A$. If $O_f \subset U$ does not
 hold, then by
 repeating the arguments of the proof of  (i) of this Lemma, a
 contradiction can be obtained. Thus, we conclude that $O_f \subset U$.
  If $f \in M_o^2 \bigcap M_{cv}$, $U \in CIFC_f$ and $c(U)=1$, then $|U \bigcap
O_f|=1$ by Lemma~\ref{lemma-completelyinvariant}(i). As $f \in
M_o^2$, $O_f \subset U$ is not possible. Therefore, $CV_f \subset U$.
\end{enumerate}
\end{proof}

\begin{proof}[Proof of Theorem~\ref{cifc}]
 \begin{enumerate}
 \item Suppose $f \in M_o^2$ and $|CIFC_f|>2$.
 Then each of the completely invariant components is simply connected.
 Further, each of them contains exactly one omitted value
 by Lemma~\ref{lemma-completelyinvariant}(i). This is not possible
 since $|O_f|=2< |CIFC_f|$ and we conclude that  $|CIFC_f| \leq 2$.
\item
Suppose $f \in M_o^{2} \bigcap M_{cv}$ and $|CIFC_f|>1$. Then each of
the completely invariant components is simply
 connected. Further, each of them contains either $CV_f$ or $O_f$ by
 Lemma~\ref{lemma-completelyinvariant} (ii).
 By Lemma~\ref{lemma-completelyinvariant}(i), $O_f$ cannot be
 contained in a single completely invariant Fatou component, which
 implies that each of these components must contain $CV_f$. This is
 not possible because $|CIFC_f| > 1$. Thus, $|CIFC_f| \leq 1$.
 \item
 Suppose $f \in M_o^{1} \bigcap M_{cv}$ and $|CIFC_f|>2$.
  Then each of the completely invariant components is simply
 connected. Lemma~\ref{lemma-completelyinvariant}(ii) gives that
 each of these components either contain $CV_f$
 or $O_f$. This cannot be true if $|CIFC_f| > 2$ and we conclude that $|CIFC_f| \leq
 2$.
\end{enumerate}
\end{proof}

\begin{proof}[Proof of Theorem~\ref{hermanring-period}]

Suppose $H$ is a $p$-periodic Herman ring and $\gamma$ is an
$f^p$-invariant curve in $H$. Let $B(H)$ denote the bounded
component of $H^c$. An essential singularity of $f^p$ cannot be in
$H$ and $B(H)=B(\gamma) \setminus H$. Therefore, it is sufficient to
prove that $B(\gamma)$ contains an essential singularity of $f^p$.
If $f^p$ is analytic on $B(\gamma)$ then $f^p(B(\gamma))=B(\gamma)$
and $f^p(B(\gamma_j))=B(\gamma_j) \subset H_j$ for
$j=1,~2,~3,...,~p-1$ where $\gamma_j=f^j(\gamma)$. This implies that
$\bigcup_{k=0}^{\infty}f^k(B(\gamma))$ is bounded, which is not
possible since $B(\gamma) $ intersects the Julia set. Therefore,
$f^p$ has at least a singularity in $B(\gamma)$. Now suppose that
all these singularities are poles. Then $f^p(B(\gamma))$ is an
unbounded domain with its boundary contained in $f^p(\gamma) =
\gamma$. Since there are $f^p$-invariant curves in $H \bigcap
B(\gamma)$, $f^p(B(\gamma))$ intersects $B(\gamma)$ and $\partial
f^p(B(\gamma))$ is properly contained in $f^p(\gamma)=\gamma$.
Consequently, the closure of $f^p(B(\gamma))=f(f^{p-1}(B(\gamma)))$
contains an omitted value of $f$. This is not possible by
Remark~\ref{allmultiplyconnected} (iii) because $f^{p-1}(B(\gamma))$
is bounded. Thus, $f^p$ has an essential singularity in $B(H)$.
\par
For $p=1$, $\infty$ is the only singularity of $f^p$ and for each
invariant Herman ring $H$, $\infty \in B(H)$. This is evidently not
possible and we conclude that $\mathcal{F}(f)$ has no invariant
Herman ring.
\par
Suppose $f$ has a single pole $w_0$ and it has a cycle of Herman
rings $\{H_0,~H_1\}$ of period $2$. Since $w_0$ is the only finite
essential singularity of $f^2$, we have $w_0 \in B(H_i)$ for
$i=0,~1$. That means the Herman rings are nested. Let $B(H_1)
\subset B(H_0)$. Take an $f^2$-invariant Jordan curve $\gamma$ in
$H_0$ and set $\gamma_1=f(\gamma)$. The set $A$ defined by
$B(\gamma) \setminus B(\gamma_1)$ does not contain the pole $w_0$
and $\partial A =\gamma \bigcup \gamma_1$ is preserved under $f$.
Thus, $f(A) =A$ and $f^n(A) =A$ for all $n$. This negates the fact
that $A \bigcap \mathcal{J}(f) \neq \emptyset$. Therefore, Herman
rings of period two do not exist for $f \in M_o$ if $f$ has only one
pole.

\end{proof}
A proof of Theorem~C using Theorem~D is given for the sake of
completeness.
\begin{proof}[Proof of Theorem~C]
Suppose that (i) is not true. Then $c(W_m) $ is finite for some $m$.
If there is a $k$ with $1 \leq c(W_k) \leq 2$, then noting that $
W_k$ is either an island or a tongue over $W_{k+1}$, it follows from
 Theorem~D that $1 \leq c(W_{k+1}) \leq 2$. Consequently,
  $1 \leq c(W_n) \leq 2$ for all $n \geq k$ and conclusion (ii) holds.
   On the other hand, if  $1 \leq c(W_k) \leq 2$ does not hold for any $k$ then in view of Theorem~D,
    we get that $f~:W_n \to  W_{n+1}$ is proper for all $n \geq m$ and by Theorem~D(i),
      $c(W_n) \geq c(W_{n+1}) \geq c(W_{n+2})...\geq 3$. Therefore, $c(W_n)=p \geq 3$
       for all large $n$ and $p$ is independent of $n$. Now, it is easy to see from Theorem ~D (i)
        that $f~:~W_n \to   W_{n+1}$ is univalent for all large $n$.
\end{proof}

\begin{proof}[Proof of Theorem~\ref{eventualconnectivity-tk}]
\begin{enumerate}
\item
If $c(W_n)=1$ for all $n \in \mathbb{N}$, then $ec(W)=1$. Suppose
$c(W_n)>1$ for some $n$. Then, $c(W_n)>1$ for all large $n$ by
Remark~\ref{remark-lemma-general3}(i). It follows from Theorem~C
that $ec(W)$ exists and is $\infty$, $2$ or $p>2$.
\item
It is clear from
Theorems~\ref{sc1},~\ref{twofatoucomponent},~\ref{onefatoucomponent},~\ref{singleton-notburied}
and ~\ref{singleton-buried}  that in each situation different from
that assumed in (i) above, each multiply connected Fatou component is
 either pre-periodic or  a wandering
domain $V$ such that $c(V_n)=1$ for all $n$. Therefore, $ec(W)=1$.

\end{enumerate}

\end{proof}


\begin{thebibliography}{99}
%
%


\bibitem{b4}
 {\bibname
I.~N. Baker}, `Wandering domains for maps of a punctured plane',
{\em Ann. Acad. Sci. Fenn.} 12 (1987) 191--198.
\bibitem{b5}
 {\bibname
I.~N. Baker}, `Domains of normality of an entire function', {\em
Fenn. Ser. A I Math.} 1 (1975) 277--283.


\bibitem{bk3}
 {\bibname
I.~N. Baker, J.~Kotus \and L.~Yinian}, `Iterates of meromorphic
functions {III}: Preperiodic domains', {\em Ergodic Theory and
Dynamical Systems} 11 (1991) 603--618.



\bibitem{berg93}
{\bibname W.~Bergweiler}, `Iteration of meromorphic functions', {\em
Bull. Amer. Math.
  Soc.} 29 (2) (1993) 151--188.
\bibitem{berg-direct}
{\bibname W.~Bergweiler \and A.~E. Eremenko},  `Direct singularities
and completely invariant domains of entire functions', {\em To
appear in Illinois J. Math. }.

\bibitem{bergere95}
{\bibname W.~Bergweiler \and A.~E. Eremenko},  `On the singularities
of the inverse to a meromorphic function of finite order', {\em
Revista Math. Iberoamericana} 11 (1995) 355--373.

\bibitem{bolsch}
 {\bibname
A. Bolsch}, `Periodic Fatou components of  meromorphic functions',
{\em Bull. London Math. Soc.} 31 (1999) 543--555.


\bibitem{transdom98}
{\bibname P.~Dom\'{i}nguez}, `Dynamics of transcendental meromorphic
functions', {\em Ann.
  Acad. Sci. Fenn. Math. } 23 (1998)  225--250.

\bibitem{res-dom-fagella}
{\bibname P.~Dom\'{i}nguez \and N. Fagella}, `Residual {J}ulia sets
of rational and
  transcendental functions', {\em Transcendental dynamics and complex analysis},
 (Cambridge University Press, 2007).
\bibitem{tanlkjk97}
{\bibname L. Keen \and J. Kotus}, `Dynamics of the family
$\lambda$tanz', {\em Conform. Geom. Dyn.} 1 (1997)  28--57.

\bibitem{tk-thesis}
{\bibname T. Nayak}, `Dynamics of certain transcendental entire and
meromorphic functions', {\em Thesis available
at~\mbox{http://www.math.sunysb.edu/cgi-bin/thesis.pl?thesis07-1}}


\bibitem{tkmgpp-unboundedsingular}
{\bibname T. Nayak \and M. G. P. Prasad}, `Iteration of certain
meromorphic functions with unbounded singular values', {\em
 Ergodic Theory and Dynamical Systems} 30 (2010) 877--891.


\bibitem{zheng8}
{\bibname T.~W. Ng, J.~H. Zheng \and Y.~Y. Choi}, `Residual {J}ulia
sets of
  meromorphic functions', {\em Math. Proc. Cambridge Philos. Soc.}
  141 (1)
  (2006) 113--126.

\bibitem{tkmgpp-tanhexp}
{\bibname M. G. P. Prasad \and T. Nayak}, `Dynamics of $\{\lambda
\tanh
  (e^z)~:~\lambda \in \mathbb{R} \setminus \{0\}\}$', {\em Discrete and Continuous Dynamical
  Systems-Series A} 19 (1)(September 2007) 121--138.
\bibitem{rippon-stallard-5}
{\bibname P. J. Rippon \and G. M. Stallard}, `On multiply connected
wandering domains of meromorphic functions', {\em J. London Math.
Soc.} 77 (2008)  405--423.
\bibitem{Zheng}
{\bibname J.~H. Zheng}, `On multiply-connected Fatou components in
iteration of meromorphic functions', {\em J. Math. Anal. Appl.} 313
(2006) 24--37.
\bibitem{zheng1}
{\bibname J.~H. Zheng}, `Singularities and limit functions in
iteration of
  meromorphic functions', {\em J. London Math. Soc.} 67 (2) (2003)
  195--207.
  \bibitem{zheng-book}
{\bibname J.~H. Zheng}, `Dynamics of meromorphic functions (in
Chinese)', {\em Tsinghua University Press} (2006).


\end{thebibliography}
\end{document}